\documentclass[a4paper,12pt,reqno]{amsart}
\usepackage{amsmath,a4,amsfonts,amssymb,amsthm,enumerate,fleqn,breqn,}
\usepackage{enumitem}
\setdescription{leftmargin=0cm,labelindent=0cm}

\usepackage{etoolbox, stmaryrd}
\patchcmd{\section}{\scshape}{\bfseries}{}{}
\makeatletter
\renewcommand{\@secnumfont}{\bfseries}
\makeatother

\theoremstyle{plain}
\newtheorem{theorem}{\bf Theorem}[section]
\newtheorem{lemma}[theorem]{\bf Lemma}
\theoremstyle{remark}
\newtheorem{definition}[theorem]{\bf Definition}

\newtheorem{example}[theorem]{\bf Example}
\newtheorem{remark}[theorem]{\bf Remark}

\title[Characterizations of Dual of $K$-frames in Quaternionic Hilbert Spaces ]{Characterizations of Dual of $K$-frames in Quaternionic Hilbert Spaces}
\author[Chander Shekhar]{Chander Shekhar}
\address{Chander Shekhar , Department of Mathematics,Indraprastha college for Women, University of Delhi, Delhi-110007,INDIA}
\email{cshekhar@ip.du.ac.in}

\begin{document}
	\maketitle \baselineskip14pt
	\begin{abstract}
		Dual of $K-$ frames in a right quaternionic Hilbert space has been recently introduced and studied by Ellouz\cite{a}.  In this paper, we study duals of $K-$ frames and prove a characterization of a $K-$ dual in terms of the canonical $K-$ dual of a $K-$ frame in the case when the range of $K$ is not necessarily closed. Moreover, we defined and studied the Approximate $K-$ Dual of a $K-$ frame in a right quaternionic Hilbert space and proved various results for their existence. 
	\end{abstract}

	\section{Introduction}
	Frames in Hilbert spaces were first introduced by J. Duffin and A.C. Schaeffer in 1952 in the context of nonharmonic Fourier series \cite{k}. Several decades later, in 1986, Daubechies, Grossmann, and Meyer revived and further developed the concept, leading to its widespread use in signal processing and harmonic analysis \cite{l}. In fact, J. Duffin and A.C. Schaeffer \cite{k} defined frame for Hilbert space $\mathbb{H}$ as follows: \\
	A countable family $\{f_k\}_{k\in\ I}\subset\mathbb{H}$ is called a frame for $\mathbb{H}$ if there exist positive constants $A, B$ such that 
	
	\begin{align}\label{1}
		A||x||^2\le\sum_{k \in I}|\langle x, f_k\rangle|^2\le B||x||^2, \quad x\in \mathbb{H}.
	\end{align}
	where $A$ and $B$ are the lower and upper bounds, respectively. 
	If only the upper inequality in \eqref{1}  is satisfied, then $\{f_k\}_{k\in\ I}$ is called a Bessel sequence. A frame $\{f_k\}_{k\in\ I}$ is said to be
	\begin{itemize}
		\item Tight if in inequality \eqref{1} , $A=B$
		\item Parseval if inequality \eqref{1} , $A=B=1$
		\item Exact if it ceases to be a frame whenever any single element is removed from the sequence $\{f_k\}_{k\in\ I}$
	\end{itemize}
	In recent years, frame theory has been extensively studied, leading to numerous generalizations. For references, see \cite{b,c,d,e,f, h,i, j,m,n,q}. In particular, Gavruta \cite{e} in 2012 introduced and explored the concept of $K-$ frames in Hilbert spaces, providing the following definition:

	Let $K$ be any bounded operator on a Hilbert space $\mathbb{H}$. A countable family $\{f_k\}_{k\in I}\subset\mathbb{H}$ is called a $K$-frame for $\mathbb{H}$ if there exist positive constants $A$ and $B$ such that 
	
	\begin{align*}
		A||K^*x||^2\le\sum_{k\in I}|\langle x, f_k\rangle|^2\le B||x||^2, \quad x\in \mathbb{H}.
	\end{align*}
	Also, $\{f_k\}_{k\in I}$ is said to be a Parseval $K-$ frame if  $||K^*x||^2=\sum_{k\in I}|\langle x, f_k\rangle|^2$.
	
	Clearly, when $K=I_{\mathbb{H}}$, the notion of a $K-$ frame reduces to that of an ordinary frame. Thus, $K-$ frames naturally extend the concept of ordinary frames.
	
	Every $K-$ frame is evidently a Bessel sequence. Hence, similar to ordinary frames, the synthesis operator for a $K-$ frame $F= \{f_k\}_{k\in I}$ can be defined as
	\begin{align*}
		T_F:\ell^2 \to \mathbb{H} ; \hspace{1cm} T_F(\{c_k\}_{k\in I})= \sum_{k\in I}c_kf_k
	\end{align*}
	
	It is a bounded linear operator, and its adjoint, known as the analysis operator, is given by
	\begin{align*}
		T_F^* : \mathbb{H} \to \ell^2; \hspace{1cm} T_F^*(x)=\{\langle x,f_k \rangle\}_{k\in I}
	\end{align*}
	
	Finally, the frame operator is given by
	\begin{align*}
		S_F : \mathbb{H}\to \mathbb{H}; \hspace{1cm} S_F(x)= T_FT_F^*(x)= \sum_{k\in I}\langle x, f_k \rangle f_k
	\end{align*}
	
	Observe that the frame operator of a $K-$ frame is not necessarily invertible, but if the range of $K$ is closed, then $S_F$ from $R(K)$ onto $S_F(R(K))$ is an invertible operator \cite{d}.
	Also 
	\begin{align*}
		B^{-1}||x|| \leq ||S_F^{-1}x|| \leq A^{-1}||K^{\dagger}||^2 ||x||, \quad x\in S_F(R(K))
	\end{align*}
	where $K^{\dagger}$ is the pseudo inverse of $K$\cite{b}.

	In \cite{b}, the pseudo-inverse of an operator $K\in B(\mathbb{H})$ with closed range is denoted by $K^{\dagger}$ and is defined as the extension of the operator $\widetilde{K}^{-1}$ to $\mathbb{H}=R(K)+R(K)^{\perp}$ with $N(K^{\dagger})=R(K)^{\perp}$ where $\widetilde{K}:=K|_{N(K)^{\perp}}:N(K)^{\perp}\to\ R(K)$.   Also pseudo inverse $K^{\dagger}$ satifies the following properties:
	\begin{enumerate}
		\item[(i)] $KK^\dagger K=K$
		\item[(ii)] $K^{\dagger}KK^{\dagger}=K^{\dagger}$
		\item[(iii)] $(KK^{\dagger})^{*}=KK^{\dagger}$
		\item[(iv)] $(K^{\dagger}K)^*=K^{\dagger}K$
	\end{enumerate}

	\section{Quaternionic Hilbert Space}
	
	The space of quaternions is given by 
	$$\mathfrak{H}=\{q_1+q_2i+q_3j+q_4k: q_!,q_2,q_3,q_4\in \mathbb{R}\}$$
	where $i^2=j^2=k^2=-1; ij=-ji=-k; jk=-kj=i; \& ki=-ik=j.$  Note that $\mathfrak{H}$ is a four-dimensional algebra with unity. 
	
	For $q\in\mathfrak{H}$ we have the following:
	\begin{itemize}
		\item Conjugate: $$\bar{q}=q_1-q_2i-q_3j-q_4k$$
		\item Modulus: $$|q|^2=q_1^2+q_2^2+q_3^2+q_4^2$$
		\item Real and Imaginary parts: $$Re(q)=q_1, \hspace{1cm}Im(q)=q_2i+q_3j+q_4k$$
		\item Inverse: $$q^{-1}=\frac{q_1-q_2i-q_3j-q_4k}{q_1^2+q_2^2+q_3^2+q_4^2}$$
	\end{itemize} 
	
	Due to the non-commutative nature of quaternions, there exist two types of quaternionic Hilbert spaces: the left quaternionic Hilbert space and the right quaternionic Hilbert space, distinguished by the placement of quaternions.
	
	In this section, we introduce some fundamental notations related to the algebra of quaternions, right quaternionic Hilbert spaces, and operators on these spaces.
	
	\begin{definition}
		A right quaternionic vector space $\mathbb{H}(\mathfrak{H})$ is a linear vector space over the field of quaternions $\mathfrak{H}$, equipped with right-scalar multiplication satisfying the following properties:
		For each $x,y \in \mathbb{H}(\mathfrak{H})$ and $p,q \in \mathfrak{H} $
		\begin{align*}
			(x+y)q&=xq+yq\\
			x(p+q)&=xp+xq\\
			x(pq)&=(xp)q
		\end{align*}
	\end{definition}
	
	\begin{definition}
		A right quaternionic inner product space $\mathbb{H}(\mathfrak{H})$  is a right quaternionic vector space together with the binary mapping $\langle . , . \rangle : \mathbb{H}(\mathfrak{H}) \times \mathbb{H}(\mathfrak{H}) \to \mathfrak{H}$ satisfying the following properties:
		\begin{itemize}
			\item $\overline{\langle x, y \rangle}= \langle y, x \rangle , \hspace{1cm} \forall x,y \in \mathbb{H}(\mathfrak{H})$
			\item $\langle x, x \rangle > 0 \hspace{0.5cm} if \hspace{0.5cm} x>0$
			\item $\langle x, y_1 +y_2 \rangle = \langle x, y_1 \rangle + \langle x, y_2 \rangle \hspace{1cm} \forall x, y_1, y_2 \in \mathbb{H}(\mathfrak{H})$
			\item $\langle x, yq \rangle = \langle x,y \rangle q$ \hspace{1cm} $\forall x,y \in \mathbb{H}(\mathfrak{H})$ and $q \in \mathfrak{H}$.
		\end{itemize}
	\end{definition}
	
	\begin{definition}\label{b}
		The right quaternionic inner product space $\mathbb{H}(\mathfrak{H})$  with the inner product $\langle . , . \rangle : \mathbb{H}(\mathfrak{H}) \times \mathbb{H}(\mathfrak{H}) \to \mathfrak{H}$ is called \textit{a right quaternionic Hilbert space} if it is complete with respect to the norm $||.|| :\mathbb{H}(\mathfrak{H}) \to \mathbb{R}^+ $ defined by $$||x||= \sqrt{\langle x, x \rangle}, \hspace{1cm} x\in \mathbb{H}(\mathfrak{H})$$ 
	\end{definition}
	
	\begin{theorem}\cite{o}
		If  $\mathbb{H}(\mathfrak{H})$ is a \textit{right quaternionic Hilbert space} then
		\begin{align*}
			{|\langle x, y \rangle|}^2 \leq \langle x, x \rangle \langle y, y \rangle , \hspace{1cm} \forall x, y \in \mathbb{H}(\mathfrak{H}).
		\end{align*}
		Moreover the norm defined in definition \ref{b} satisfy the following properties:
		\begin{itemize}
			\item $||xq||=||x||\hspace{0.1cm} |q|, \hspace{1cm} \forall x \in \mathbb{H}(\mathfrak{H})$ and $q \in \mathfrak{H}.$
			\item $||x+y|| \leq ||x|| + ||y||, \hspace{1cm} \forall x, y \in \mathbb{H}(\mathfrak{H})$
			\item $||x||=0$ for some $x \in \mathbb{H}(\mathfrak{H})$ then $x=0$
		\end{itemize}
	\end{theorem}
	
	\begin{example}\cite{o}
		For the non- commutative field of quaternions $\mathfrak{H}$, define the space $\ell^2(\mathfrak{H})$ by 
		\begin{align*}
			\ell^2(\mathfrak{H})= \Bigg\{ \{q_i \}_{i\in I} \subset \mathfrak{H} : \sum_{i\in I} ||q_i||^2 < + \infty \Bigg \}
		\end{align*}
		Clearly $\ell^2(\mathfrak{H})$ is a \textit{right quaternionic Hilbert space} with respect to the inner product defined as 
		\begin{align*}
			\langle p, q \rangle = \sum_{i\in I} \bar{p_i} q_i , \hspace{1cm} p=\{p_i\}_{i\in I}, q=\{q_i\}_{i\in I} \in \ell^2(\mathfrak{H})
		\end{align*}
	\end{example}
	\begin{definition}\cite{o}
		Let $\mathbb{H}(\mathfrak{H})$ be a right quaternionic Hilbert space and $M$ be a subset of $\mathbb{H}(\mathfrak{H})$. Then define the sets 
		\begin{itemize}
			\item $M^{\perp}=\{x\in \mathbb{H}(\mathfrak{H}) : \langle x, y \rangle =0 \hspace{0.5cm} \forall y \in M\}$
			\item $\langle M \rangle$ be the right linear subspace of $\mathbb{H}(\mathfrak{H})$  consisting of all finite right linear combinations of elements of $M$.
		\end{itemize}

	\end{definition}
	
	\begin{theorem}\cite{o}\label{c}
		Let $\mathbb{H}(\mathfrak{H})$ be a right quaternionic Hilbert space and $N$ be a subset of $\mathbb{H}(\mathfrak{H})$ such that, $z_1, z_2 \in N, \langle z_1, z_2 \rangle =0$ if $z_1 \neq z_2$ and $\langle z_1, z_1 \rangle =1$. Then the following conditions are equivalent:
		\begin{enumerate}
			\item [(a)] For every $x, y \in \mathbb{H}(\mathfrak{H})$, the series $\sum_{z\in N} \langle x, z \rangle \langle z, y \rangle$ converges absolutely and 
			\begin{align*}
				\langle x, y \rangle= \sum_{z\in N} \langle x, z \rangle \langle z, y \rangle
			\end{align*}
			\item [(b)] For every $x\in \mathbb{H}(\mathfrak{H}), ||x||^2 = \sum_{z\in N} |\langle x, z \rangle |^2$
			\item [(c)] $N^{\perp}=0$
			\item [(d)] $\langle N \rangle =0$
		\end{enumerate}
	\end{theorem}
	\begin{definition}\cite{o}
		Every quaternionic Hilbert space $\mathbb{H}(\mathfrak{H})$ admits a subset $N$, called Hilbert basis or orthonormal basis of $\mathbb{H}(\mathfrak{H})$, such that for $z_1, z_2 \in N, \langle z_1, z_2 \rangle =0$ if $z_1 \neq z_2$ and $\langle z_1, z_1 \rangle =1$ and satisfies all the conditions of Theorem \ref{c}.
	\end{definition}
	\begin{definition}\cite{p}
		Let $\mathbb{H}(\mathfrak{H})$ be a right quaternionic Hilbert space and $T$ be an operator on $\mathbb{H}(\mathfrak{H})$. Then $T$ is said to be 
		\begin{itemize}
			\item right linear if $T(xp+yq)=T(x)p + T(y)q$, for all $x, y \in \mathbb{H}(\mathfrak{H})$ and $p,q \in \mathfrak{H}$.
			\item bounded if there exists $K \geq 0$ such that $||T(x)|| \leq K ||x||$, for all $x \in \mathbb{H}(\mathfrak{H})$.            
		\end{itemize}
	\end{definition}
	
	\begin{definition}\cite{p}
		Let $\mathbb{H}(\mathfrak{H})$ be a right quaternionic Hilbert space and $T$ be an operator on $\mathbb{H}(\mathfrak{H})$. Then the adjoint $T^*$ of $T$ is defined by 
		\begin{align*}
			\langle T^*x, y \rangle = \langle x, Ty \rangle ,\hspace{1cm} \forall x, y \in \mathbb{H}(\mathfrak{H})
		\end{align*}
		Further, $T$ is said to be self adjoint if $T=T^*$.
	\end{definition}
	
	\begin{theorem}\cite{p}
		Let $\mathbb{H}(\mathfrak{H})$ be a right quaternionic Hilbert space and $T_1$, $T_2$ be bounded right linear operators on $\mathbb{H}(\mathfrak{H})$. Then 
		\begin{enumerate}
			\item [(a)] $T_1 + T_2$ and $T_1 T_2$ are also bounded right linear operators on $\mathbb{H}(\mathfrak{H})$. Moreover 
			\begin{align*}
				||T_1 + T_2|| \leq ||T_1 || + ||T_2 ||, \hspace{1cm}
				||T_1 T_2 || = ||T_1|| \hspace{0.1cm} ||T_2||
			\end{align*}
			\item [(b)] $\langle T_1x, y \rangle= \langle x, T_1^*y \rangle , \hspace{1cm} \forall x,y \in \mathbb{H}(\mathfrak{H})$
			\item [(c)] $(T_1 + T_2)^* = T_1^* + T_2^*$
			\item [(d)] $(T_1 T_2)^* = T_2^* T_1^*$
			\item [(e)] $(T_1^*)^* = T$
			\item [(f)] $I^* = I$ where $I$ is an identity operator on $\mathbb{H}(\mathfrak{H})$
			\item [(g)] If $T$ is an invertible operator then $(T^{-1})^* = (T^*)^{-1}$
		\end{enumerate}
	\end{theorem}
	
	\begin{theorem}\cite{o}
		Let $\mathbb{H}(\mathfrak{H})$ be a right quaternionic Hilbert space and $T$ be an operator on $\mathbb{H}(\mathfrak{H})$. If $T \geq 0$, then there exists a unique bounded right linear operator on $\mathbb{H}(\mathfrak{H})$, denoted by $\sqrt{T}$, such that $\sqrt{T} \geq 0$ and $\sqrt{T} \sqrt{T} = T$. Furthermore $\sqrt{T}$ commutes with every operator which commute with $T$. Also, if $T$ is invertible and self adjoint then $\sqrt{T}$ is invertible and self adjoint.
	\end{theorem}
	
	In the recent years frames theory has been studied in the quaternionic Hilbert spaces. One may refer to \cite{a, j, o,p,r,s,t,u,v}
	\section{$K-$ duals of a $K-$ frame in quaternionic Hilbert spaces}
	
	Throughout this paper, $\mathbb{H}(\mathfrak{H})$ is a seperable right quaternionic Hilbert space, $I$ a countable index set, $\mathbb{B}(\mathbb{H}(\mathfrak{H}))$ denotes the set of all bounded right linear operators on $\mathbb{H}(\mathfrak{H})$. In addition, the range of $K\in\mathbb{B}(\mathbb{H}(\mathfrak{H}))$ is denoted by $R(K)$ and the orthogonal projection of $\mathbb{H}(\mathfrak{H})$ onto a closed subspace $M\subseteq \mathbb{H}(\mathfrak{H})$ is denoted by $P_M$. Also, for $K\in\mathbb{B}(\mathbb{H}(\mathfrak{H}))$, $K^{\dagger}$ denotes the pseudo inverse of $K$.
	
	\begin{definition}
		Let $\{f_k\}_{k\in I}$ be a $K-$ frame for a right quaternionic Hilbert space $\mathbb{H}(\mathfrak{H})$ then a Bessel sequence $\{g_k\}_{k\in I}$ is said to be a $K-$ dual of $\{f_k\}_{k\in I}$ if 
		\begin{align}\label{2}
			Kx= \sum_{k\in I} f_k \langle x, g_k\rangle, \quad x\in \mathbb{H}(\mathfrak{H})
		\end{align}
	\end{definition}
	\begin{remark}\cite{a}
		If $\{f_k\}_{k\in I}$ be a $K-$ frame for a right quaternionic Hilbert space $\mathbb{H}(\mathfrak{H})$ then there exists a Bessel sequence $\{g_k\}_{k\in I}$ satisfying \eqref{2} and \begin{align}
			K^*x= \sum_{k\in I} g_k \langle x, f_k\rangle, \quad x\in \mathbb{H}(\mathfrak{H})
		\end{align}
	\end{remark}
	
	\begin{lemma}\cite{a}\label{a}
		Let $\{f_k\}_{k\in I}$ and $\{g_k\}_{k\in I}$ be two Bessel sequences for the right quaternionic Hilbert space $\mathbb{H}(\mathfrak{H})$ satisfying \eqref{2}. Then $\{f_k\}_{k\in I}$ and $\{g_k\}_{k\in I}$ are $K$ and $K^*-$ frames, respectively.  
	\end{lemma}
	\begin{proof}
		For $x\in \mathbb{H}(\mathfrak{H})$, we have
		\begin{align*}
			||Kx||^4&=|\langle Kx, Kx \rangle|^2 \\ 
			&=\bigg|\biggl< \sum_{k\in I}f_k \langle x, g_k \rangle , Kx \biggr> \bigg|^2  \\
			&\leq \sum_{k\in I}|\langle Kx, f_k \rangle|^2 \sum_{k\in I} |\langle x, g_k \rangle|^2 \\
			&\leq B||Kx||^2 \sum_{k\in I} |\langle x, g_k \rangle|^2 
		\end{align*}
		where $B$ is an upper bound of the bessel sequence $\{f_k\}_{k\in I}$.
		Hence
		\begin{align*}
			B^{-1}||Kx||^2 \leq \sum_{k\in I} |\langle x, g_k \rangle|^2 
		\end{align*}
		
		This shows that $\{g_k\}_{k\in I}$ is a $K^*-$ frame with lower bound $B^{-1}$. In the same manner, we can obtain the lower bound of $\{f_k\}_{k\in I}$ by repeating the above argument for $K^*$ instead of $K$.
	\end{proof}
	
	\begin{example}
		Let $\mathbb{H}(\mathfrak{H})=\mathfrak{H}^4$ and $\{e_1,e_2,e_3,e_4\}$ be orthonormal basis of $\mathbb{H}(\mathfrak{H})$. Define $K:\mathbb{H}(\mathfrak{H})\to\mathbb{H}(\mathfrak{H})$ by 
		\begin{align*}
			Ke_1=e_1=f_1\\
			Ke_2=e_1=f_2\\
			Ke_3=e_2=f_3\\
			Ke_4=e_3=f_4
		\end{align*}
		then $\{f_1,f_2,f_3,f_4\}$ is a $K-$ frame. 
		
		Now, we compute 
		
		\begin{align*}
			Kx&=\sum_{k=1}^4 Ke_k \langle x, e_k\rangle \\
			&=\sum_{k=1}^4 f_k \langle x ,e_k\rangle 
		\end{align*}
		Hence, $\{e_k\}_{k=1}^4$ is a $K-$ dual of $\{f_k\}.$ Moreover, we have
		\begin{align*}
			K^*x&=\sum_{k=1}^4 e_k \langle K^*x, e_k\rangle \\
			&=\sum_{k=1}^4 e_k \langle x, Ke_k\rangle \\
			&=\sum_{k=1}^4 e_k \langle x, f_k\rangle 
		\end{align*}
		Therefore, $||K^*x||^2=\sum_{k=1}^4|\langle x, f_k|^2$ .i.e $\{f_k\}$ is a Parseval $K-$ frame. 
		
		Take  $h_1=(1,2,0,0), h_2=(0,-1,0,0), h_3=(0,0,1,0), h_4=(0,0,0,1)$. So, we compute 
		\begin{align*}
			\sum_{k=1}^4 f_k \langle x, h_k\rangle &= e_1 \langle x, e_1+2e_2\rangle +e_1 \langle x, -e_2\rangle +e_2\langle x, e_3\rangle +e_3\langle x, e_4\rangle \\
			&= \sum_{k=1}^4 f_k\langle x, e_k\rangle \\&=Kx
		\end{align*}
		Therefore, $\{h_k\}_{k=1}^4$ is also a $K-$ dual of $\{f_k\}_{k=1}^4$ which is different from $\{e_k\}_{k=1}^4$. 
	\end{example}
	
	\begin{theorem}
		Let $\mathbb{H}(\mathfrak{H})$ be a right quaternionic Hilbert space and $K$ be a bounded right linear operator on $\mathbb{H}(\mathfrak{H})$ with closed range. Let $\{f_k\}_{k\in I}$ be a $K-$ frame for $\mathbb{H}(\mathfrak{H})$ with lower and upper bounds $A$ and $B$, respectively. Let S be the corresponding frame operator. Then $\{K^* S^{-1}P_{S(R(K))}f_k\}_{k\in I}$ is a $K-$ dual of $\{P_{R(K)}f_k\}_{k\in I}$ with bounds $B^{-1}$ and $A^{-1}||K||^2||K^{+}||^2$, respectively.
	\end{theorem}
	
	\begin{proof}
		We know that $S$ is a bounded right linear operator and invertible on $R(K)$. 
		Also $S=TT^*$ is self adjoint on $\mathbb{H}(\mathfrak{H})$ and $S^{-1}S |_{R(K)}=I_{R(K)}$. Therefore $\{K^* S^{-1}P_{S(R(K))}f_k\}_{k\in I}$ is a Bessel sequence in $\mathbb{H}(\mathfrak{H})$. Also, we have 
		\begin{align*}
			Kx=(S^{-1}S)^*Kx&=S^*(S^{-1})^*Kx\\
			&=S^*P_{S(R(K))}(S^{-1})^*Kx\\
			&=\sum_{k\in I } P_{R(K)}f_k\langle  P_{S(R(K))}(S^{-1})^*Kx , f_k\rangle \\
			&=\sum_{k\in I }P_{R(K)}f_k\langle  x , K^* S^{-1}P_{S(R(K))}f_k\rangle 
		\end{align*}
		Hence, $\{K^* S^{-1}P_{S(R(K))}f_k\}_{k\in I}$ is a $K-$ dual of $\{P_{R(K)}f_k\}_{k\in I}$ with lower bound $B^{-1}$ by lemma \ref{a}. Also, we have for $x\in R(K)$
		\begin{align*}
			||(S^{-1})^*x||^2&=\langle S^{-1}(S^{-1})^*x,x\rangle\\
			&\le A^{-1}||K^{\dagger}||^2||(S^{-1})^*x|| ||x||
		\end{align*}
		So, we have 
		\begin{align*}
			\sum_{k\in I}|\langle x,  K^* S^{-1}P_{S(R(K))}f_k\rangle|^2 &=\sum_{k\in I}|\langle (S^{-1})^*Kx, f_k\rangle|^2\\ &=\langle S(S^{-1})^*Kx, (S^{-1})^*Kx \rangle \\ &= \langle Kx, (S^{-1})^*Kx \rangle \\
			&\le A^{-1}||K||^2||K^{\dagger}||^2||x||^2, x\in \mathbb{H}(\mathfrak{H}).
		\end{align*}
		Hence the results.
	\end{proof}
	
	\begin{remark}
		We define the canonical $K-$ dual as the $K-$ dual $\{K^* S^{-1}P_{S(R(K))}f_k\}_{k\in I}$ of $\{P_{R(K)}f_k\}_{k\in I}$ obtained in Theorem 2.3. Clearly this canonical $K-$ dual is same as the ordinary canonical dual if $\{f_k\}_{k\in I}$ is an ordinary frame.
	\end{remark}

	Next theorem is the characterization of all $K-$ duals of $\{P_{R(K)}f_k\}_{k\in I}$ in terms of the canonical $K-$ dual of  $\{P_{R(K)}f_k\}_{k\in I}$.
	
	\begin{theorem}
		Let $K$ be a bounded operator on a right quaternionic Hilbert space $\mathbb{H}(\mathfrak{H})$ with closed range, and let $\{f_k\}_{k\in I}$ be a $K-$ frame for $\mathbb{H}(\mathfrak{H})$. Then, $\{g_k\}_{k \in I}$ is a $K-$ dual of $\{P_{R(K)}f_k\}_{k\in I}$ if and only if there exists a bounded operator $M: \mathbb{H}(\mathfrak{H}) \to \ell ^2(\mathfrak{H})$ such that $P_{R(K)}T_FM=0$ and
		\begin{align*}
			g_k= K^* S_F^{-1}P_{S_F(R(K))}f_k + M^*e_k, \quad k \in I
		\end{align*}
		where $\{e_k\}_{k \in I}$ is the standard orthonormal basis of $\ell ^2(\mathfrak{H})$ and $T_F$ is the synthesis operator of $\{f_k\}_{k\in I}$.
	\end{theorem}
	
	\begin{proof}
		
		Suppose $M:\mathbb{H}(\mathfrak{H}) \to \ell ^2(\mathfrak{H})$ be a bounded operator such that $P_{R(K)}T_FM=0$.
		We can see that $\{   g_k\}_{k\in I}$= $\{K^* S_F^{-1}P_{S_F(R(K))}f_k + M^*e_k\}_{k\in I}$ is a Bessel sequence; in fact, 
		\begin{align*}
			\sum_{k\in I}|\langle x, K^* S_F^{-1}P_{S_F(R(K))}f_k + M^*e_k\rangle |^2&\le 2\left(\sum_{k\in I} |\langle x, K^* S_F^{-1}P_{S_F(R(K))}f_k\rangle|^2 +\sum_{k\in I}|\langle x,M^*e_k\rangle |^2\right)\\
			&\le 2(A^{-1}||K||^2||K^{\dagger}||^2+||M||^2)||x||^2,
		\end{align*}
		for all $x\in \mathbb{H}(\mathfrak{H})$, where $A$ is a lower bound for $\{f_k\}_{k\in I}$. Also
		\begin{align*}
			\sum_{k\in I} P_{(R(K)}f_k\langle x, K^* S_F^{-1}P_{S_F(R(K))}f_k +M^*e_k\rangle  &=\sum_{k\in I}P_{(R(K)}f_k\langle x, K^* S_F^{-1}P_{S_F(R(K))}f_k\rangle \\ &\ \ \ \ +\sum_{k\in I}P_{(R(K)}f_k\langle x, M^*e_k\rangle \\
			&=Kx+ P_{R(K)}T_FMx\\
			&=Kx
		\end{align*}
		Hence, $\{   g_k\}_{k\in I}$= $\{K^* S_F^{-1}P_{S_F(R(K))}f_k + M^*e_k\}_{k\in I}$ is a $K-$ dual of $\{P_{R(K)}f_k\}_{k\in  I}$. Conversely, let $\{g_k\}_{k\in I}$ be a $K-$ dual of $\{P_{R(K)}f_k\}_{k\in I}$. Define
		\begin{align*}
			M=T_G^*-T_F^*P_{S_F(R(K))}(S_F^{-1})^*K
		\end{align*}
		Then $M: \mathbb{H}(\mathfrak{H}) \to \ell ^2(\mathfrak{H})$ is a bounded operator and
		\begin{align*}
			P_{R(K)}T_FMx&=P_{R(K)}T_FT_G^*x-(S_F)^*(S_F^{-1})^*Kx\\
			&=\sum_{k\in I}P_{R(K)}f_k \langle x, g_k \rangle -Kx=0,
		\end{align*}
		for every $x\in \mathbb{H}(\mathfrak{H})$. Also
		\begin{align*}
			K^* S_F^{-1}P_{S_F(R(K))}f_k + M^*e_k&=   K^* S_F^{-1}P_{S_F(R(K))}f_k+T_Ge_k-K^*S_F^{-1}P_{S_F(R(K))}T_Fe_k\\
			&= K^* S_F^{-1}P_{S_F(R(K))}f_k+T_Ge_k-K^*S_F^{-1}P_{S_F(R(K))}f_k\\
			&=g_k
		\end{align*}
		for all $k\in I$. Hence the result.
	\end{proof}
	
	In the following theorem, we characterize all $K-$ duals of a $K-$ frame when the range of $K \in B(\mathbb{H}(\mathfrak{H}))$ is not necessarily  closed.
	
	\begin{theorem}
		Let $\{f_k\}_{k\in I}$ be a $K-$ frame for $\mathbb{H}(\mathfrak{H})$. Then, $\{g_k\}_{k \in I}$ is a $K-$ dual of $\{f_k\}_{k\in I}$ if and only if there exists a bounded operator $M:\ell ^2(\mathfrak{H}) \to  \mathbb{H}(\mathfrak{H})$ such that $T_FM^*=K$ and 
		\begin{align*}
			g_k=Me_k, \quad k \in I
		\end{align*}
		where $\{e_k\}_{k \in I}$ is the standard orthonormal basis of $\ell ^2(\mathfrak{H})$ and $T_F$ is the synthesis operator of $\{f_k\}_{k\in I}$.
		
	\end{theorem}
	
	\begin{proof}
		Let $\{g_k\}_{k\in I}$ be a $K-$ dual of the $K-$ frame $\{f_k\}_{k\in I}$. Then 
		\begin{align*}
			Kx=\sum_{k\in I}f_k\langle x, g_k\rangle 
		\end{align*}
		Suppose $T_G$ is the synthesis operator of $\{g_k\}_{k\in I}$. Put $M=T_G$. Then $T_FM^*x=T_F(\{\langle x, g_k\rangle\})=\sum_{k \in I}f_k \langle x, g_k\rangle= Kx,\quad x\in \mathbb{H}(\mathfrak{H}) $. This gives $T_FM^*=K$. Also $Me_k=T_Ge_k=g_k$, for all $k\in I$.  
		
		Conversely, if $M:\ell ^2(\mathfrak{H}) \to  \mathbb{H}(\mathfrak{H})$ is a bounded linear operator such that $T_FM^*=K$.  Write $g_k=Me_k$. Then $\{g_k\}_{k\in I}$ is a Bessel sequence and we have 
		\begin{align*}
			\sum_{k \in I}f_k\langle x, g_k\rangle=\sum_{k \in I}f_k\langle M^*x, e_k\rangle=T_FM^*x=Kx 
		\end{align*}
		Thus $\{g_k\}_{k\in I}$ is a $K-$ dual of the $K-$ frame $\{f_k\}_{k\in I}$.
	\end{proof}

	In \cite{g} it has been proved that $\{f_k\}_{k\in I}$ and $\{g_k\}_{k\in I}$ are Bessel sequences satisfying \eqref{2} and $K$ is a bounded operator on $\mathbb{H}(\mathfrak{H})$ with closed range, then there exists a sequence $\{h_k\}_{k\in I}=\{(K^{\dagger}|_{R(K)})^*g_k\}_{k\in I}$ 
	such that 
	\begin{align}\label{3}
		x= \sum_{k\in I}f_k\langle x,h_k\rangle, \quad  x\in R(K) 
	\end{align}
	Also, $x= \sum_{k\in I}h_k\langle x,f_k\rangle, \quad  x\in R(K) $. In the following result, we give a characterization of sequence $\{h_k\}_{k\in I}$ in \eqref{3}

	\begin{theorem}
		Let $K:\mathbb{H}(\mathfrak{H})\to\mathbb{H}(\mathfrak{H})$ be a bounded linear operator with closed range, $\{f_k\}_{k\in I}$ be a $K-$ frame for $\mathbb{H}(\mathfrak{H})$ and $\{g_k\}_{k\in I}$ be  a $K-$ dual of $\{f_k\}_{k\in I}$.  Then $\{h_k\}_{k\in I}=\{(K^{\dagger}|_{R(K)})^*g_k\}_{k\in I}$ satisfies 
		\begin{align*}
			x= \sum_{k\in I}f_k\langle x,h_k\rangle, \quad  x\in R(K)   
		\end{align*}
		if and only if $\{h_k\}_{k\in I}=\{Me_k\}_{k\in I}$, where $\{e_k\}_{k\in I}$ is the orthonormal basis of $\ell^2(\mathbb{H}(\mathfrak{H}))$ and $M:\ell^2(\mathbb{H}(\mathfrak{H}))\to\ R(K)$ is a left inverse of $T_F^*|_{R(K)}$.  
	\end{theorem}
	
	\begin{proof}
		Suppose $\{h_k\}_{k\in I}=\{(K^{\dagger}|_{R(K)})^*g_k\}_{k\in I}$  such that
		\begin{align*}
			x= \sum_{k\in I}f_k\langle x,h_k\rangle, \quad  x\in R(K)   
		\end{align*}
		Take $M=T_H$. Then $M:\ell^2(\mathbb{H}(\mathfrak{H}))\to\ R(K) $ such that  $\{h_k\}_{k\in I}=\{Me_k\}_{k\in I}$ and  $T_HT_F^*|_{R(K)}=I$. 
		
		Conversely, if  $M:\ell^2(\mathbb{H}(\mathfrak{H}))\to\ R(K) $ such that  $\{h_k\}_{k\in I}=\{Me_k\}_{k\in I}$ and  $MT_F^*|_{R(K)}=I$, then
		\begin{align*}
			\sum_{k\in I}f_k\langle x, h_k\rangle= \sum_{k\in I}f_k\langle M^*x, e_k\rangle= T_FM^*x=x.
		\end{align*}
		
	\end{proof}

	\section{Approximate $K-$ Dual}
	\begin{definition}
		Let $\{f_k\}_{k\in I}$ be a $K-$ frame for a right quaternionic Hilbert space $\mathbb{H}(\mathfrak{H})$ then a Bessel sequence $\{g_k\}_{k\in I}$ is called an approximate $K-$ dual of $\{f_k\}_{k\in I}$ if 
		\begin{align}
			|| K- T_FT^*_G  || <1
		\end{align}
	\end{definition}

	\begin{example}
		Let $\mathbb{H}(\mathfrak{H})=\mathfrak{H}^3$ be the right quaternionic Hilbert space and $\{e_1, e_2, e_3\}$ be the standard orthonormal basis of $\mathbb{H}(\mathfrak{H})$. 
		
		Define $K:\mathbb{H}(\mathfrak{H})\to\mathbb{H}(\mathfrak{H})$ by 
		\begin{align*}
			Ke_1&=e_1\\
			Ke_2&=\frac{1}{2}e_2\\
			Ke_3&=0
		\end{align*}
		Also
		\begin{align*}
			Kx=e_1\langle x, e_1 \rangle  + e_2\frac{1}{2} \langle x, e_2 \rangle  
		\end{align*}
		We can easily check that $F=\{ \frac{1}{2} e_2, e_1, e_3\}$ is a $K-$ frame for $\mathbb{H}(\mathfrak{H})$ and $G=\{2e_2, e_1, e_3\}$ is an approximate $K-$ dual of $F$ as 
		
		\begin{align*}
			||(K-T_F T_G^*)(x)||\leq \frac{1}{2}||x||, \quad x\in \mathbb{H}(\mathfrak{H})
		\end{align*}
		
	\end{example}

	\begin{theorem}
		Let $K$ be a bounded operator on a right quaternionic Hilbert space $\mathbb{H}(\mathfrak{H})$ with closed range, $\||K^{\dagger}||=1$ and let $F=\{f_k\}_{k\in I}$ be a $K-$ frame for $\mathbb{H}(\mathfrak{H})$. Let $G=\{g_k\}_{k\in I}$ be an approximate $K-$ dual of $\{f_k\}_{k\in I}$. Then 
		\begin{enumerate}
			\item [(a)] $\{\Psi^{-1}P_{R(K)}f_k\}_{k\in I}$ is a $K-$ frame for $\mathbb{H}(\mathfrak{H})$ with a $K-$ dual $\{K^*(K^{\dagger})^*g_k\}_{k\in I}$ where $\Psi=P_{R(K)}T_F T_G^*K^{\dagger}$ and $K^{\dagger}$ be the pseudo inverse of $K$.
			
			\item [(b)] $\{(K^{\dagger})^*g_k\}_{k\in I}$ is a $K-$ frame for $\mathbb{H}(\mathfrak{H})$ with a $K-$ dual $\{K^*\Psi ^{-1}P_{R(K)}f_k\}_{k\in I}$.
		\end{enumerate}
	\end{theorem}
	\begin{proof}
		(a).  Write $h_k=(K^{\dagger})^*g_k$, so we compute
		\begin{align*}
			\Bigg\|x-\sum_{k\in I}f_k \langle x, h_k \rangle \Bigg\| = \Bigg\|KK^{\dagger}x -\sum_{k\in I}f_k\langle K^{\dagger}x, g_k\rangle \Bigg\|<||x||, \ x\in R(K).
		\end{align*}
		This implies 
		\begin{align*}
			\Bigg\|I| _{R(K)}-P_{R(K)}T_{F}{T^*_H}|_{R(K)}\Bigg\|<1.
		\end{align*}
		Therefore $\Psi=P_{R(K)}T_{F}{T^*_H}|_{R(K)}$ is invertible on $R(K)$. Hence, we have 
		\begin{align*}
			Kx=\Psi^{-1}\Psi Kx&= \sum_{k\in I}\Psi^{-1}P_{R(K)}f_k\langle Kx, h_k\rangle\\
			&= \sum_{k\in I}\Psi^{-1}P_{R(K)}f_k\langle x, K^* (K^{\dagger})^*g_k\rangle, \ x\in\mathbb{H}(\mathfrak{H}).
		\end{align*}
		Hence, $\{\Psi^{-1}P_{R(K)}f_k\}_{k\in I}$ is a $K-$ frame for $\mathbb{H}(\mathfrak{H})$ with a $K-$ dual $\{K^*(K^{\dagger})^*g_k\}_{k\in I}$ by Lemma \ref{a}.
		
		(b). For each $x\in \mathbb{H}(\mathfrak{H})$, we have 
		\begin{align*}
			Kx=\Psi^*(\Psi^*)^{-1}Kx&= \sum_{k\in I} h_k \langle (\Psi^*)^{-1} Kx, P_{R(K)f_k}\rangle\\
			& = \sum_{k\in I}(K^{\dagger})^*g_k\langle x, K^*\Psi^{-1} P_{R(K)}f_k\rangle
		\end{align*}
		Hence,  $\{(K^{\dagger})^*g_k\}_{k\in I}$ is a $K-$ frame for $\mathbb{H}(\mathfrak{H})$ with a $K-$ dual $\{K^*\Psi ^{-1}P_{R(K)}f_k\}_{k\in I}$.
	\end{proof}
	
	\begin{theorem}
		Let $K$ be a bounded linear operator on a right quaternioinic Hilbert space $\mathbb{H(\mathfrak{H})}$ with closed range. Let $\{f_k\}_{k\in I}$ be a $K-$ frame for $\mathbb{H(\mathfrak{H})}$,  $P_{R(K)}$ be the projection onto the range $K$ and $\{g_k\}_{k\in I}$ be a $K-$ dual of $\{P_{R(R)}f_k\}_{k\in I}$. \\
		If $||T_FT_G^*||<1$, then $\{g_k\}_{k\in I}$ will be approximate $K-$ dual of $\{f_k\}_{k\in I}$.
		
	\end{theorem}
	\begin{proof}
		Since $\{g_k\}_{k\in I}$ be a $K-$ dual of $\{P_{R(K)}f_k\}_{k\in I}$ then $K=P_{R(K)}T_FT_G^*$. Therefore 
		\begin{align*}
			||K-T_FT_G^*||&=||P_{R(K)}T_FT_G^*-T_FT_G^*||\\
			&\le||1-P_{R(K)}|||T_FT_G^*||<1.
		\end{align*}
		Hence, $\{g_k\}_{k\in I}$ will be an approximate $K-$ dual of $\{f_k\}_{k\in I}$.
	\end{proof}
	
	\begin{theorem}
		Let $K$ a bounded linear operator on a right quaternioinic Hilbert space $\mathbb{H(\mathfrak{H})}$ with closed range. Let $\{f_k\}_{k\in I}$ be a $K-$ frame for $\mathbb{H(\mathfrak{H})}$,  $P_{R(K)}$ be the projection onto the range of $K$ and $\{g_k\}_{k\in I}$ be an approximate $K-$ dual of $\{P_{R(K)}f_k\}_{k\in I}$. If $||(I-P_{R(K)})T_FT_G^*||\le 1-||K-P_{R(K)}T_FT_G^*||$ then $\{g_k\}_{k\in I}$ is an approximate $K-$ dual frame of $\{f_k\}_{k\in I}$.
	\end{theorem}
	
	\begin{proof}
		Since $\{g_k\}_{k\in I}$ is an approximate $K-$ dual of $\{P_{R(K)}f_k\}_{k\in I}$, therfore $||K-P_{R(K)}T_FT_G^*||<1$. Also 
		\begin{align*}
			||K-T_FT_G^*||=||K-P_{R(K)}T_FT_G^*||+||P_{R(K)}T_FT_G^*-T_FT_G^*||<1.
		\end{align*}
		Hence, $\{g_k\}_{k\in I}$ is an approximate $K-$ dual frame of $\{f_k\}_{k\in I}$.
	\end{proof}

\end{document}